\documentclass[12pt]{amsart}
\usepackage{fullpage,amssymb}
\newtheorem{theorem}{Theorem}[section]
\newtheorem{lemma}[theorem]{Lemma}
\newtheorem{corollary}[theorem]{Corollary}
\newtheorem{proposition}[theorem]{Proposition}
\newtheorem*{notation*}{Notation}
\newtheorem*{p*}{Proposition~\ref{h.s.o.p}}
\theoremstyle{definition}
\newtheorem{definition}[theorem]{Definition}

\newtheorem{remark}[theorem]{Remark}
\newcommand{\M}{{\operatorname{Mat}}}

\newcommand{\N}{{\mathbb Z}_{\geq 0}}
\newcommand{\Z}{{\mathbb Z}}

\newcommand{\K}{K}
\newcommand{\Sl} {{\operatorname{SL} }}

\newcommand{\SL}{{\operatorname{SL}}}

\newcommand{\Hom}{\operatorname{Hom}}
\newcommand{\kar}{\operatorname{char}}

\newcommand{\Q}{{\mathbb Q}}
\newcommand{\GL}{\operatorname{GL}}
\newcommand{\Rep}{\operatorname{Rep}}
\newcommand{\SI}{\operatorname{SI}}

\newcommand{\Proj}{\operatorname{Proj}}

\usepackage{multirow}
\usepackage{amsmath}
\usepackage{comment}
\usepackage{tikz}
\usepackage{mathtools}
\usepackage{arydshln}

\title{Degree bounds for semi-invariant rings of quivers}

\author{Harm Derksen and Visu Makam}

\thanks{The first author was supported by NSF grant DMS-1302032 and the second author was supported by NSF grant DMS-1361789}

\begin{document}

\maketitle

\begin{abstract}
We use recent results on matrix semi-invariants to give degree bounds on generators for the ring of semi-invariants for quivers with no oriented cycles.
\end{abstract}

\section{Introduction}
In \cite{DM}, we studied the left-right action of $\SL_n \times \SL_n$ on $m$-tuples of $n \times n$ matrices. Among other things, we proved bounds for the degree of generators for the invariant ring. This ring of invariants can be seen as the ring of semi-invariants for the $m$-Kronecker quiver for the dimension vector $(n,n)$.  
In this paper, we obtain bounds for the degree of generators for the ring of semi-invariants for a quiver with no oriented cycles.

\subsection{Quiver representations} A quiver is just a directed graph. Formally, a quiver $Q$ is a pair $(Q_0,Q_1)$, where $Q_0$ is a finite set of vertices, and $Q_1$ is a finite set of arrows. For each arrow $a \in Q_1$, we denote its head and tail by $ha$ and $ta$ respectively. We fix an infinite field $K$. A representation $V$ of $Q$ over $K$ is a collection of finite dimensional $K$-vector spaces $V(x)$, $x\in Q_0$ together with a collection of $K$-linear maps $V(a):V(ta)\to V(ha)$, $a\in Q_1$. The dimension vector of $V$ is the function $\alpha:Q_0\to \N$ such that
$\alpha(x)=\dim V(x)$ for all $x \in Q_0$.

Let $\M_{p,q}$ denote the set of $p \times q$ matrices over $K$. For a dimension vector $\alpha\in \N^{Q_0}$, we define its representation space by:
$$\Rep(Q,\alpha)=\prod_{a\in Q_1}\M_{\alpha(ha),\alpha(ta)}.$$ 
 If $V$ is a representation with dimension vector $\alpha$
and we identify $V(x)\cong K^{\alpha(x)}$ for all $x$, then $V$ can be viewed as an element of $\Rep(Q,\alpha)$.
Consider the group $\GL(\alpha)=\prod_{x\in Q_0} \GL_{\alpha(x)}$ and its subgroup $\Sl(\alpha)=\prod_{x\in Q_0}\Sl_{\alpha(x)}$.
The group $\GL(\alpha)$ 
acts on $\Rep(Q,\alpha)$ by:
$$
(A(x)\mid x\in Q_0)\cdot (V(a)\mid a\in Q_1)=(A(ha)V(a)A(ta)^{-1}\mid a\in Q_1).
$$
For $V\in \Rep(Q,\alpha)$, choosing a different basis means acting by the group $\GL(\alpha)$. The $\GL(\alpha)$-orbits in $\Rep(Q,\alpha)$ correspond to isomorphism classes of representations of dimension $\alpha$.

\subsection{Invariants for quiver representations}
The group $\GL_n$ acts by simultaneous conjugation on $\M_{n,n}^m$, the space of $m$-tuples of $n \times n$ matrices. Artin conjectured that in characteristic $0$, the invariant ring is generated by traces of words in the matrices. Procesi proved this conjecture, and also showed that invariants of degree $\leq 2^n -1$ generate the ring of invariants (see \cite{Procesi}). It was shown by Razmyslov that invariants of degree $\leq n^2$ suffice (see \cite[final remark]{Raz}). A concise account of the above results can also be found in \cite{Formanek}. 

Le Bruyn and Procesi generalized the results to arbitrary quivers. They proved that the ring of invariants $\K[\Rep(Q,\alpha)]^{\GL(\alpha)}$ is generated by traces along oriented cycles. Using the aforementioned bound, they showed that the invariants of degree $\leq N^2$ generate the ring, where $N = \sum_i \alpha_i$ (see \cite{LP}).

\subsection{Semi-invariants for quiver representations}
From Le Bruyn and Procesi's results described above, we see that a quiver with no oriented cycles has no non-trivial invariants. However, the ring of semi-invariants $\SI(Q,\alpha) = \K[\Rep(Q,\alpha)]^{\SL(\alpha)}$ could still be interesting. 

A multiplicative character of the group $\GL_\alpha$ is of the form
$$
\chi_\sigma:(A(x)\mid x\in Q_0)\in \GL_\alpha\mapsto \prod_{x\in Q_0}\det(A(x))^{\sigma(x)}\in K^\star,
$$
where $\sigma:Q_0\to \Z$ is called the weight of the character $\chi_\sigma$. Define 
$$\SI(Q,\alpha)_\sigma=\{f\in K[\Rep(Q,\alpha)]\mid \forall A\in \GL(\alpha)\ A\cdot f=\chi_\sigma(A) f\}.$$
Then the ring of semi-invariants has a weight space decomposition 
 $$\SI(Q,\alpha)=\bigoplus_{\sigma}\SI(Q,\alpha)_\sigma.$$

 If $\sigma\cdot \alpha=\sum_{x\in Q_0}\sigma(x)\alpha(x)\neq 0$, then $\SI(Q,\alpha)_\sigma=0$. Assume that $\sigma\cdot \alpha=0$. We can write $\sigma=\sigma_+-\sigma_-$ where $\sigma_+(x)=\max\{\sigma(x),0\}$ and $\sigma_-(x)=\max\{-\sigma(x),0\}$. Define $|\sigma|_\alpha = \sigma_+\cdot \alpha=\sigma_-\cdot \alpha$.

Now we define an $n\times n$ linear matrix
$$A:\bigoplus_{x\in Q_0}V(x)^{\sigma_+(x)}\to \bigoplus_{x\in Q_0}V(x)^{\sigma_-(x)}
$$
where each block $\Hom(V(x),V(y))$ is of the form $t_1V(p_1)+\cdots+t_rV(p_r)$, $t_1,t_2,\dots,t_r$ are indeterminates and $p_1,p_2,\dots,p_r$ are all paths from $x$ to $y$.
We use different indeterminates for the different blocks, so the  linear matrix has $m=\sum_{x\in Q_0}\sum_{y\in Q_0} \sigma_+(x)b_{x,y}\sigma_-(y)$ indeterminates where $b_{x,y}$ is the number of paths from $x$ to $y$. We can write $A=t_1X_1+\cdots+t_mX_m$ with $X_1,\dots,X_m\in \M_{n,n}$.
We have the following result (see~\cite[Corollary 3]{DW}, \cite{DZ} and  \cite{SVd}).
\begin{theorem} \label{si}
The space $\SI(Q,\alpha)_\sigma$ is spanned by $\det(t_1X_1+\cdots+t_mX_m)$ with $t_1,\dots,t_m\in K$.
\end{theorem}

The ring $\SI(Q,\alpha)$ is the ring of invariants for the action of the reductive group $\SL_\alpha$ on the vector space $\Rep(Q,\alpha)$. The ring of regular functions on $\Rep(Q,\alpha)$ is a polynomial ring, and has a natural grading by total degree. The group $\SL_\alpha$ acts on each graded piece, and thus the ring $\SI(Q,\alpha)$ inherits the grading:
$$\SI(Q,\alpha) = \bigoplus_{d \in \Z_{\geq 0}} \SI(Q,\alpha)_{[d]}.$$
Note that we use $\SI(Q,\alpha)_{[d]}$ to denote the $d^{th}$ graded piece with respect to the total degree grading, whereas we use $\SI(Q,\alpha)_\sigma$ to denote the weight space corresponding to the weight $\sigma$.
Theorem~\ref{si} gives a determinantal description for the semi-invariants of a given weight. However, such a description of the invariants in terms of the total degree grading is not readily available. On the other hand, several results in Computational Invariant Theory apply to the total degree grading, and not directly to the weight space decomposition. The null cone $\mathcal{N} \subseteq \Rep(Q,\alpha)$ is defined as the zero set of the non-constant homogeneous invariants. From the description in Theorem~\ref{si}, it follows that the null cone is the zero set of the semi-invariants for all nonzero weights. The null cone is an important tool in Computational Invariant Theory. 

Recall that $||\alpha||_1 = \sum\limits_{i \in Q_0} |\alpha_i|$ and $||\alpha||_2 = (\sum\limits_{i \in Q_0} \alpha_i^2)^{1/2}$. We prove:

\begin{theorem} \label{null cone bound}
Let $Q = (Q_0,Q_1)$ be a quiver with no oriented cycles, and let $|Q_0| = n$. Then the null cone for the action of $\SL_\alpha$ on $\Rep(Q,\alpha)$ is defined by semi-invariants for nonzero weights $\sigma$ such that $|\sigma|_\alpha \leq \displaystyle \frac{||\alpha||_1^{2n}}{4(n-1)^{2n-2}}$.
\end{theorem}

In characteristic $0$, bounds on the degree of the invariants defining the null cone can be translated into bounds for the degree of generating invariants.
\begin{theorem} \label{main}
Let $Q = (Q_0,Q_1)$ be a quiver with no oriented cycles, and let $|Q_0| = n$. Assume $\kar K =0$ and let $r$ be the Krull dimension of $\SI(Q,\alpha)$. The ring $\SI(Q,\alpha)$ is generated by semi-invariants of weights $\sigma$ with $$|\sigma|_\alpha \leq   \displaystyle \frac{3rn^2 ||\alpha||_1^{4n}}{128(n-1)^{4n-4}}.$$
\end{theorem}

Note that $\dim(\SI(Q,\alpha)) \leq \dim \Rep(Q,\alpha)$, which depends on $Q_0$ and $Q_1$. We show that using a theorem of Weyl, we can give a bound that depends only on $n = |Q_0|$ and $\alpha$. 

\begin{corollary} \label{independent}
Let $Q = (Q_0,Q_1)$ be a quiver with no oriented cycles, and let $|Q_0| = n$ Assume $\kar K = 0$. The ring $\SI(Q,\alpha)$ is generated by semi-invariants of weights $\sigma$ with $$|\sigma|_\alpha \leq \displaystyle \frac{3}{256} \left(||\alpha||_1^2 - ||\alpha||_2^2\right) \frac{n^2 ||\alpha||_1^{4n}}{(n-1)^{4n-4}}.$$
\end{corollary}

Even though our bounds are not polynomial in $n = |Q_0|$, we give an example to show that it is not possible to obtain general bounds that are polynomial. Indeed consider the quiver $Q_n$ shown below. 

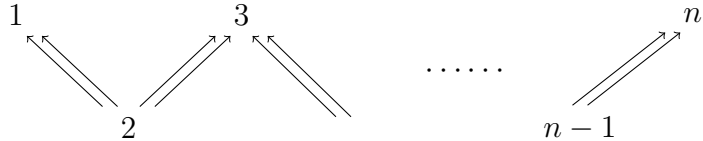
\begin{center}
 \begin{tikzpicture}

\path node (x) at (0,0) []{$1$};
\path node (y) at (1.5,-1.5) []{$2$};

\draw[->] ([xshift=0.1 cm]y.north west) -- ([xshift=0.1 cm]x.south east) ;
\draw[->] ([xshift=-0.1 cm]y.north west) -- ([xshift=-0.1 cm]x.south east) ;

\path node (z) at (3,0) {$3$};
\draw[->] ([xshift=0.1 cm]y.north east) -- ([xshift=0.1 cm]z.south west) ;
\draw[->] ([xshift=-0.1 cm]y.north east) -- ([xshift=-0.1 cm]z.south west) ;

\path node (w) at (4.5,-1.5) {};

\draw[->] ([xshift=0.1 cm]w.north west) -- ([xshift=0.1 cm]z.south east) ;
\draw[->] ([xshift=-0.1 cm]w.north west) -- ([xshift=-0.1 cm]z.south east) ;

\path node at (6,-0.75) {$\cdots\cdots$};

\path node (a) at (7.5,-1.5) {$n-1$};
\path node (b) at (9,0) {$n$};
\draw[->] ([xshift=0.1 cm]a.north) -- ([xshift=0.1 cm]b.south west) ;
\draw[->] ([xshift=-0.1 cm]a.north) -- ([xshift=-0.1 cm]b.south west) ;

\end{tikzpicture}
\end{center}

\begin{proposition} \label{expbound}
For the quiver $Q_n$, and dimension vector $\alpha = (2,3,\dots,3,1)$, the semi-invariants of weights $\sigma$ with $|\sigma|_\alpha < 2^n - 2$ do not define the null cone, and hence do not generate $\SI(Q_n, \alpha)$.
\end{proposition}

\subsection{Organization}
In Section~\ref{LA}, we describe a result in linear algebra that will be very useful to get bounds. We recall relevant results from Computational Invariant Theory in Section~\ref{CIT} and some recent results from \cite{DM} in Section~\ref{previous results}. In Section~\ref{SC}, we recall King's criterion for semi-stability and stability. We then use the King's criterion and the result from Section~\ref{LA} to get bounds for the null cone in Section~\ref{BNC}. We use the bounds on the null cone to generate bounds for the degree of generators in Section~\ref{BGS}, and in Section~\ref{RD}, we remove the dependence on $\dim(\SI(Q,\alpha))$. Finally in Section~\ref{lowerbound}, we prove Proposition~\ref{expbound}.

\section{A result from linear algebra} \label{LA}
For any vector $w = (w_1,w_2,\dots,w_n) \in \Q^n$, recall that  
$$
||w||_1 = |w_1| + |w_2| + \dots + |w_n|,
$$ and
$$
||w||_2 = (w_1^2 + w_2^2 + \dots + w_n^2)^{1/2}.
$$

We have the inequalities 
$$
||w||_2 \leq ||w||_1 \leq \sqrt{n} ||w||_2.
$$

Let  $\vec{v}_1,\vec{v}_2,\dots,\vec{v}_{n-1} \in \Z_{\geq 0}^n$ be linearly independent over $\Q$, with $\vec{v}_1 + \vec{v}_2 + \dots + \vec{v}_{n-1} = \vec{v} \in \Z_{\geq 0}^n$. Considering each $\vec{v}_i$ as a row vector, we can write a $(n-1) \times n$ matrix $M$ whose $i^{th}$ row is $\vec{v}_i$. Since the $\vec{v}_i$ are linearly independent over $\Q$, the rank of this matrix is $n-1$. Hence it has a $1$-dimensional kernel. The following proposition bounds the smallest nonzero integral vector on this $1$-dimensional kernel:

\begin{proposition}\label{linear-algebra}
Let $\vec{v}_1,\vec{v}_2,\dots,\vec{v}_{n-1}, \vec{v}$ and $M$ be as above. Then there is a nonzero integral vector $\vec{u} = (u_1,u_2,\dots,u_n) \in \operatorname{Ker}(M)$ such that $|u_i| \leq \displaystyle \left(\frac{||\vec{v}||_1}{n-1}\right)^{n-1}$
\end{proposition}

\begin{proof}

Let $\widehat M(i)$ denote the $(n-1) \times (n-1)$ minor of $M$ obtained by removing the $i^{th}$ column. Then define $u_i = (-1)^i \widehat M(i)$. It is clear that $\vec{u} = (u_1,u_2,\dots,u_n)$ is an integral vector and that it is in the kernel of $M$. Further, we have 
\begin{align*} |u_i| & \leq ||\vec{v}_1||_2\cdot ||\vec{v}_2||_2 \cdots ||\vec{v}_{n-1}||_2 \\
  & \leq  ||\vec{v}_1||_1\cdot ||\vec{v}_2||_1 \cdots ||\vec{v}_{n-1}||_1 \\
  & \leq \left(\frac{ ||\vec{v}_1||_1 +  ||\vec{v}_2||_1 +  \cdots + ||\vec{v}_{n-1}||_1}{n-1}\right)^{n-1}\\
  &= \left(\frac{||\vec{v}||_1}{n-1}\right)^{n-1} \\
\end{align*}

\end{proof}

\section{Computational Invariant Theory} \label{CIT}
Let $V$ be a rational representation of a linearly reductive group $G$. Then $\K[V]$ is a polynomial ring, and has a natural grading by total degree. The ring of invariants $\K[V]^G$ inherits this grading. Further, we know that it is finitely generated since $G$ is reductive (see~\cite{Hilbert1, Hilbert2, Nagata, Haboush}). 

\begin{definition}
We define $\beta(\K[V]^G)$ to be the smallest integer $D$ such that invariants of degree $\leq D$ generate $\K[V]^G$.
\end{definition}

A set of homogeneous invariants $f_1,f_2,\dots,f_r$ is called a homogeneous system of parameters (hsop) if they are algebraically independent, and $\K[V]^G$ is a finite module over $\K[f_1,f_2,\dots,f_r]$. In particular, the number of invariants in any hsop must be equal to the Krull dimension of $\K[V]^G$. The invariant ring $\K[V]^G$ is Cohen-Macaulay by the Hochster-Roberts theorem, which implies that $\K[V]^G$ is in fact a free module over any hsop. 

\begin{definition}{\textbf{Null Cone :}}
The null cone $\mathcal{N}$ is the zero set of all homogeneous invariant polynomials of positive degree
$$\mathcal{N} = \{ v \in V | f(v) = 0 \  for\  all \ f \in \K[V]^G_+ \}.$$ 
\end{definition}

A set of algebraically independent homogeneous invariants $f_1,f_2,\dots,f_r$ form a hsop if and only if their zero set coincides with the null cone (see \cite[Lemma~2.4.5]{DK}). Kempf showed that the Hilbert series of $\K[V]^G$ is a rational function of degree $\leq 0$ (see \cite{Kempf}). From this, one can deduce that if $\deg(f_i) = d_i$, then $\beta(\K[V]^G) \leq d_1 + d_2 + \dots + d_r$. In \cite{Knop1,Knop2} Knop further proved that the degree of the Hilbert series is in fact $\leq -r$ if $G$ is a semisimple connected group in characteristic $0$ (see also \cite[Theorem~2.6.2]{DK}).

\begin{definition}
We define $\gamma(\K[V]^G)$ to be the smallest integer $D$ such that the non-constant homogeneous invariants of degree $\leq D$ define the null cone.
\end{definition}

Popov argued that one could use bounds for $\gamma(\K[V]^G)$ to get bounds for $\beta(\K[V]^G)$. Hence one can use Kempf and Knop's results on the Hilbert series to get bounds for $\beta(\K[V]^G)$ given bounds on $\gamma(\K[V]^G)$. The first author in \cite{Derksen1} improved these bounds to one that is polynomial in $\gamma(\K[V]^G).$

\begin{theorem} [\cite{Derksen1}] \label{beta-gamma}
For a rational representation $V$ of a linearly reductive group $G$, we have
$$\beta(\K[V]^G) \leq \max \{2,\textstyle\frac{3}{8}r (\gamma(\K[V]^G))^2\}.$$
\end{theorem} 

Note that in characteristic $0$, linearly reductive groups are precisely the reductive groups. A more detailed treatment of the above results, as well as several other techniques for finding degree bounds can be found in \cite{DK}, see also \cite{Popov1,Popov2}.

\section{Degree bounds on matrix semi-invariants} \label{previous results}
In \cite{DM}, we studied the left-right action of $\SL_n \times \SL_n$ on $m$-tuples on $n \times n$ matrices. Let $R(n,m) = \K[\M_{n,n}^m]^{\SL_n \times \SL_n}$ be the ring of invariants. Studying linear subspaces of matrices, and the behaviour of ranks in tensor blow-ups, we proved:

\begin{theorem} [\cite{DM}] \label{DM-main}
We have the following bounds for invariants defining the null cone:
$$
 n\lfloor\sqrt{n^2-1}\rfloor \leq \gamma(R(n,m)) \leq n(n-1).
$$ 
\end{theorem}

In characteristic 0, we show the existence of a hsop in degree $n(n-1)$, and using the bounds on Hilbert series given by Knop, we get bounds for $\beta(R(n,m))$.

\begin{theorem} [\cite{DM}]
Let $\kar K = 0$. We have:
\begin{enumerate}
\item $\beta(R(n,m)) \leq mn^4$;
\item For all $m$, $\beta(R(n,m)) \leq n^6$.
\end{enumerate}
\end{theorem}

The second result follows from the first using Weyl's theorem on polarization of invariants (see \cite[Section~7.1, Theorem~A]{KP}).

Given a character $\sigma:Q_0 \to \Z$, we can consider the subring of semi-invariants $\SI(Q,\alpha,\sigma) = \bigoplus_{d \in \Z_{\geq 0}}\SI(Q,\alpha)_{d\sigma}.$ The projective variety $\Proj(\SI(Q,\alpha,\sigma))$, if non-empty, is a moduli space for $\alpha$-dimensional representations of $Q$ (see \cite{King}). 
A representation $V\in \Rep(Q,\alpha)$ is called $\sigma$-semistable if there exists an semi-invariant $f\in \SI(Q,\alpha)_{d\sigma}$ with $f(V)\neq 0$ (see~\cite{King}). The results on matrix semi-invariants mentioned above readily generalize to the subrings $\SI(Q,\alpha,\sigma)$. 

\begin{proposition} [\cite{DM}] \label{existssemi}
Let $Q = (Q_0,Q_1)$ be a quiver with no oriented cycles. Let $\sigma \in \Z^{Q_0}$ be a weight. Then 
\begin{enumerate}
\item If $V$ is $\sigma$-semistable, and $d\geq |\sigma|_\alpha -1$, then there exists a semi-invariant $f\in \SI(Q,\alpha)_{d\sigma}$ with $f(V)\neq 0$;
\item If $\kar K = 0$, then the ring $\SI(Q,\alpha,\sigma)$ is generated in degree $\leq |\sigma|_\alpha^5$.
\end{enumerate}
\end{proposition}

Even though we know the bounds for these subrings, they do not immediately give bounds on entire ring of semi-invariants, since there is no universal bound for $|\sigma|_\alpha$. Such a bound does not exist even if we restrict to indivisible dimension vectors. In the next two sections, we rectify this by showing that it suffices to consider only a finite subset of the dimension vectors.

\section{Stability conditions and the null cone} \label{SC}
Fix a quiver $Q = (Q_0,Q_1)$ with no oriented cycles and let $|Q_0| = n$. There is a criterion for deciding $\sigma$-semistability of a representation in terms of the dimension vectors of subrepresentations due to King (see \cite{King}). We use the conventions in \cite{DW2}. Given a weight $\sigma$ and a dimension vector $\beta$, we define $\sigma(\beta) = \sum_{i \in Q_0} \sigma_i\beta_i$.

\begin{theorem} [\cite{King}] \label{stability}
Let $Q = (Q_0,Q_1)$ be a quiver with no oriented cycles, $\alpha$ be a dimension vector, and $\sigma$ be a weight. Then we have: 
\begin{enumerate}
\item A representation $V \in \Rep(Q,\alpha)$ is $\sigma$-semistable if $\sigma(\underline{\dim} V) = 0$ and $\sigma(\underline{\dim} W) \leq 0$ for all subrepresentations $W \subset V$;
\item A representation $V \in \Rep(Q,\alpha)$ is $\sigma$-stable if $\sigma(\underline{\dim} V) = 0$ and $\sigma(\underline{\dim} W) < 0$ for all proper subrepresentations $0 \neq W \subsetneq V$.
\end{enumerate}
\end{theorem}

The set of $\sigma$-semistable representations form an abelian subcategory of the category of finite dimensional representations of a quiver $Q$. The simple objects in the category are precisely the stable representations. If $V$ is $\sigma$-semistable and $\sigma(\underline{\dim} W) = 0$ for some non-zero proper subrepresentation $W$ of $V$, then $W$ and $V/W$ are also $\sigma$-semistable. In fact, we have a Jordan-H\"{o}lder filtration $0 = V_0 \subsetneq V_1 \subsetneq \dots \subsetneq V_m = V$. The composition factors $V_i/V_{i-1}$ are unique upto rearrangement and isomorphism. Further these composition factors are $\sigma$-stable representations. We can define $\operatorname{gr}_{\sigma}(V) = \bigoplus\limits_i V_i/V_{i-1}.$

\begin{remark}
Let $d \in \Z_{>0}$. From Theorem~\ref{stability}, it follows that a representation $V$ is $\sigma$-semistable (resp. stable) if and only if $V$ is $d\sigma$-semistable (resp. stable). Hence, in particular, we have $\operatorname{gr}_\sigma(V) = \operatorname{gr}_{d\sigma}(V)$.
\end{remark}

\begin{lemma} \label{obvious}
A representation $V \in \Rep(Q,\alpha)$ is not in the null cone if and only if there exists a nonzero weight $\sigma$ such that $V$ is $\sigma$-semistable.
\end{lemma}
 
\begin{proof}
We have already remarked that the null cone is the zero set of the semi-invariants of nonzero weights. Thus if a representation $V$ is not in the null cone, then there is a semi-invariant $f \in \SI(Q,\alpha)_\sigma$ for some nonzero weight $\sigma$ such that $f(V) \neq 0$. Consequently for this $\sigma$, $f$ is $\sigma$-semistable. Conversely, if $V \in \Rep(Q,\alpha)$ is $\sigma$-semistable for some nonzero weight $\sigma$, then there is an invariant $f \in \SI(Q,\alpha)_{d\sigma}$ such that $f(V) \neq 0$ for some $d \in \Z_{>0}$. Hence $V$ is not in the null cone. 
\end{proof}

\section{Bounds for the null cone} \label{BNC}
Given a representation $V \in \Rep(Q,\alpha)$, we denote by $\operatorname{C} (V)$, the set of weights $\sigma$ for which $V$ is $\sigma$-semistable, i.e, 
$$
\operatorname{C} (V) = \{ \sigma | V \mbox{ is $\sigma$-semistable}\}.
$$ 
Notice that $\operatorname{C}(V) \subset \Z^{Q_0}$ is cut out by a linear equation $\sigma(\alpha) = 0$ and by linear inequalities $\sigma(\underline{\dim} W) \leq 0$ for proper subrepresentations $W$ of $V$. Let $L$ be an extremal ray of $\operatorname{C}(V)$. It is clear that this extremal ray is defined by degenerating a subset of the linear inequalities to equalities. Hence, there exists a set of subrepresentations $W_i$, $i \in I$ of $V$ such that $\Q L$ is defined by $\sigma(\alpha) = 0$ and $\sigma(\underline{\dim} W_i) = 0$ for $i \in I$. 

Let $\sigma$ be the smallest integral weight on $L$. Then since $\sigma \in \operatorname{C}(V)$, we have that $V$ is $\sigma$-semistable. We can look at a Jordan-H\"older series $0 = V_0 \subsetneq V_1 \subsetneq \dots \subsetneq V_m = V$. The composition factors $V_i/V_{i-1}$ are $\sigma$-stable representations, and let $\alpha_i := \underline{\dim} V_i/V_{i-1}$. Thus, in particular we have $\sigma(\alpha_i) = 0$ for all $i$. It is easy to see that the inequalities $\sigma(\alpha_i) = 0$ will define $\Q L$. However, some of these linear equalities may be redundant. Since these equalities define a $1$-dimensional subspace, we can find a subset of the $\alpha_i$'s of size $n-1$, say $\{\beta_1,\beta_2,\dots,\beta_{n-1}\}$, such that $\sigma(\beta_i) = 0$ for $i = 1,2,\dots, n-1$ defines $\Q L$.

\begin{proposition} \label{abc}
Given a representation $V \in \Rep(Q,\alpha)$ such that $\operatorname{C} (V)$ is non-empty, $V$ is $\sigma$-semistable for some weight $\sigma$ such that each coordinate of $\sigma$ has size $\leq \displaystyle \left(\frac{||\alpha||_1}{n-1}\right)^{n-1}$
\end{proposition}

\begin{proof}
Observe that $\Q L$ is the kernel of the $(n-1) \times n$ matrix whose rows are the dimension vectors $\beta_i$. Since $\sigma$ is the smallest integral vector in $\Q L$, we have that each coordinate of $\sigma$ is bounded by $\displaystyle \left(\frac{||\alpha||_1}{n-1}\right)^{n-1}$, by Proposition~\ref{linear-algebra}.
\end{proof}

We can now translate this into a bound for $|\sigma|_\alpha$.

\begin{corollary} \label{sigma-bound}
Given a representation $V \in \Rep(Q,\alpha)$ such that $\operatorname{C} (V)$ is non-empty, $V$ is $\sigma$-semistable for some weight $\sigma$ such that $$|\sigma|_\alpha \leq \displaystyle \left(\frac{||\alpha||_1}{n-1}\right)^{n-1}\left(\frac{||\alpha||_1}{2}\right) = \frac{||\alpha||_1^n}{2(n-1)^{n-1}} .$$
\end{corollary}

\begin{proof}
If every coordinate $|\sigma_i| \leq M$ for some $M$, then we have $|\sigma_i| \alpha_i \leq M \alpha_i$. Note further that since $\sigma(\alpha) = 0$, we have $\sum\limits_{i=1}^n  |\sigma_i| \alpha_i = \sigma_+ \cdot \alpha + \sigma_{-} \cdot \alpha = 2 |\sigma|_\alpha$. Thus $|\sigma|_\alpha \leq \frac{1}{2} M ||\alpha||_1$.
\end{proof}

\begin{proof}[Proof of Theorem~\ref{null cone bound}]
Given $V \in \Rep(Q,\alpha)$ which is not in the null cone, we have that $C(V)$ is nonzero by Lemma~\ref{obvious}. Hence there exists some $\sigma$ with $|\sigma|_\alpha \leq \frac{||\alpha||_1^n}{2(n-1)^{n-1}} $, such that $V$ is $\sigma$-semistable. Then by the first part of Proposition~\ref{existssemi}, there is an semi-invariant $f \in \SI(Q,\alpha)_{d\sigma}$ that does not vanish on $V$ for each $d \geq |\sigma|_\alpha - 1$. Observe that $|d\sigma|_\alpha = d|\sigma|_\alpha$. Taking $d = |\sigma|_\alpha$ gives the required conclusion. 
\end{proof}

\begin{remark} \label{extremal rays}
It might seem very wasteful to find bounds using an extremal ray $L$, as it is very likely that smaller weights lie in the interior of $\operatorname{C}(V)$. However, observe that if $\sigma$ is an integral weight on an extremal ray $L$ of $\operatorname{C}(V)$, then for $\operatorname{gr}_\sigma(V)$ we have $\operatorname{C}(\operatorname{gr}_\sigma(V)) = L$. Hence these extremal rays cannot be avoided. 
\end{remark}

\section{Bounds for generating semi-invariants} \label{BGS}
The ring $\SI(Q,\alpha)$ has two natural gradings. We have the weight space decomposition $\SI(Q,\beta) = \bigoplus_\sigma \SI(Q,\alpha)_\sigma$. We also have the natural grading inherited from viewing $\K[\Rep(Q,\alpha)]$ as a polynomial ring. While the weight space decomposition is the more interesting one, all the results from Computational Invariant Theory hold for the latter grading. In the previous section, we found bounds for invariants defining the null cone in terms of the weight space decomposition. In order to use Theorem~\ref{beta-gamma}, we must switch to the latter grading.


\begin{lemma}\label{gradings}
Let $f \in \SI(Q,\alpha)_\sigma$, then its homogeneous components are non-trivial only for degrees between $|\sigma|_\alpha$ and $n |\sigma|_\alpha$.
\end{lemma}

\begin{proof}
A set of semi-invariants spanning $f \in \SI(Q,\alpha)_\sigma$ was given in Theorem~\ref{si}. A semi-invariant in this set is given by the determinant of a matrix, whose size is $|\sigma|_\alpha$. The matrix is described in block form, where each block defines a linear combinations of paths between two different vertices. Such paths have length at least $1$ and at most $n$. Hence the entries of this matrix are polynomials whose homogeneous components are non-trivial only for degrees between $1$ and $n$. 
\end{proof}

The above lemma can then be used to convert the bounds given in Theorem~\ref{null cone bound} with respect to weight spaces to one in the total degree grading. 

\begin{corollary}
The null cone for the action of $\SL_\alpha$ on $\Rep(Q,\alpha)$ is defined by homogeneous invariants of degree $\leq \displaystyle \frac{n ||\alpha||_1^{2n}}{4(n-1)^{2n-2}} $, i.e.,  
$$
\gamma(\SI(Q,\alpha)) \leq \displaystyle \frac{n ||\alpha||_1^{2n}}{4(n-1)^{2n-2}}.
$$

\end{corollary}

Finally, we can apply Theorem~\ref{beta-gamma} to get:
\begin{corollary}\label{poly grade bounds}
Assume $\kar K =0$ and let $r = \dim(\SI(Q,\alpha))$. The ring of semi-invariants $\SI(Q,\alpha)$ is generated by invariants of degree $\leq \displaystyle \frac{3}{8} r \left(  \frac{n ||\alpha||_1^{2n}}{4(n-1)^{2n-2}}\right)^2$.
\end{corollary}

\begin{proof}[Proof of Theorem~\ref{main}]
This follows from Lemma~\ref{gradings} and Corollary~\ref{poly grade bounds}.
\end{proof}

\begin{remark}
The bounds given in Theorem~\ref{main} depend on $\dim(\SI(Q,\alpha))$. Kac gave a formula (see \cite{Kac}) for $\dim(\SI(Q,\alpha))$ in terms of the canonical decomposition. There is in fact an efficient algorithm to compute the canonical decomposition due to the first author and Weyman, see \cite{DW1}. More importantly, as remarked in the introduction, $\dim(\SI(Q,\alpha))$ is bounded by $\dim(\Rep(Q,\alpha)) = \sum_{a \in Q_1} \alpha(ha) \alpha(ta)$. 
\end{remark}

\section{Removing dependence on $\dim \SI(Q,\alpha)$} \label{RD}
The bounds in Theorem~\ref{main} depend on $|Q_0| = n$,  $\alpha$ and $\dim(\SI(Q,\alpha))$. Note that $\dim(\SI(Q,\alpha))$ depends on $Q_1$. We now show how one can use Weyl's theorem on polarization of invariants to remove the dependence on $\dim(\SI(Q,\alpha))$, and get a bound which is purely in terms of $|Q_0| = n$ and $\alpha$.

Given a quiver $Q = (Q_0,Q_1)$ with no oriented cycles, we can label the vertices $1,2,\dots,n$ so that for every arrow, $ta < ha$. Let $n(i,j)$ denote the number of arrow with tail $i$ and head $j$. Fix a dimension vector $\alpha = (\alpha_1,\alpha_2,\dots,\alpha_n)$. Now, observe that 
$$
\Rep(Q,\alpha) = \bigoplus_{i < j} \M_{\alpha_j,\alpha_i}^{ n(i,j)}.
$$  

Observe further that each $\M_{\alpha_j,\alpha_i}$ is a representation of $\GL_\alpha$ as well as $\SL_\alpha$. Observe that $\dim \M_{\alpha_j,\alpha_i} = \alpha_i\alpha_j$. Hence, as a consequence of Weyl's theorem on polarization of invariants (see \cite[II.5, Theorem~2.5A]{Weyl} and \cite[Section~7.1,Theorem~A]{KP}), we can obtain the semi-invariant ring $\SI(Q,\alpha)$ by the process of polarization from $\K[\bigoplus\limits_{i < j} \M_{\alpha_j,\alpha_i}^{ \alpha_i\alpha_j} ]^{\SL_\alpha}$. See also \cite[Theorem~0.1]{DKW} for a version that is better suited to our situation. In other words, for the purposes of finding a bound on the generating invariants, we can assume $n(i,j) = \alpha_i\alpha_j$.

 Define a quiver $\widetilde{Q}$ whose vertex set is $1,2,\dots,n$, and has $\alpha_i\alpha_j$ arrows from $i$ to $j$. The above discussion can be summarized as follows:

\begin{proposition}
Assume $\kar K = 0$, then we have 
$$
\beta(\SI(Q,\alpha)) \leq \beta(\SI(\widetilde{Q},\alpha)).
$$
\end{proposition}

\begin{proof}[Proof of Corollary~\ref{independent}]
For $\widetilde{Q}$, we have:
\begin{align*}
\dim(\SI(\widetilde{Q},\alpha)) & \leq \dim \Rep(\widetilde{Q},\alpha) \\
 &  = \sum\limits_{i < j} \alpha_i \alpha_j \\ 
  & = \displaystyle \frac{||\alpha||_1^2 - ||\alpha||_2^2}{2}.\\
\end{align*}
Now, use this bound for $r$ in Corollary~\ref{poly grade bounds}, and apply Lemma~\ref{gradings}. 
\end{proof}

\section{Exponential lower bound} \label{lowerbound}
We first recall some results on the $2$-Kronecker quiver, the quiver with 2 vertices $x$ and $y$ and two arrows $a,b$ from $x$ to $y$. 

\begin{center}
\begin{tikzpicture}

\path node (x) at (0,0) []{$x$};
\path node (y) at (2,0) []{$y$};

\draw[->] ([yshift=0.1 cm]x.east) -- node[midway,above]{$a$}([yshift=0.1 cm]y.west) ;
\draw[->] ([yshift=-0.1 cm]x.east) -- node[midway,below]{$b$} ([yshift=-0.1 cm]y.west) ;

\end{tikzpicture}
\end{center}

We look at two particular indecomposable representations. The representation 
\begin{center}
\begin{tikzpicture}

\path node at (-0.75,0) {$V = $};
\path node (x) at (0,0) []{$\K$};
\path node (y) at (2,0) []{$\K^2$};

\draw[->] ([yshift=0.1 cm]x.east) -- node[midway,above]{$\left(\begin{smallmatrix} 1 \\0 \end{smallmatrix}\right)$}([yshift=0.1 cm]y.west) ;
\draw[->] ([yshift=-0.1 cm]x.east) -- node[midway,below] {$\left(\begin{smallmatrix} 0 \\1 \end{smallmatrix}\right)$} ([yshift=-0.1 cm]y.west) ;

\end{tikzpicture}
\end{center}

is an indecomposable representation for the dimension vector $(1,2)$. It is easy to check that $V$ is $\sigma$-semistable precisely when $\sigma \in \Z_{> 0} (2,-1)$. Similarly, the representation
\begin{center}
\begin{tikzpicture}

\path node at (-0.75,0) {$W = $};

\path node (x) at (0,0) []{$\K^2$};
\path node (y) at (2,0) []{$\K$};

\draw[->] ([yshift=0.1 cm]x.east) -- node[midway,above]{$\left(\begin{smallmatrix} 1& 0 \end{smallmatrix}\right)$}([yshift=0.1 cm]y.west) ;
\draw[->] ([yshift=-0.1 cm]x.east) -- node[midway,below] {$\left(\begin{smallmatrix} 0& 1 \end{smallmatrix}\right)$} ([yshift=-0.1 cm]y.west) ;

\end{tikzpicture}
\end{center}

is an indecomposable representation for the dimension vector $(2,1)$. Once again, it is easy to check that $W$ is $\sigma$-semistable precisely when $\sigma \in \Z_{> 0} (1,-2)$.

\begin{proof} [Proof of Proposition~\ref{expbound}]
Consider the quiver $Q_n$, and observe that the odd vertices are sources and the even vertices are sinks. For any $i \in \{1,2\dots,n-1\}$, one of $i$ and $i+1$ is a source and the other is a sink. Let $\psi_i$ be the embedding of the $2$-Kronecker quiver, that maps the vertices to $i$ and $i+1$, with source begin mapped to source and sink to sink. Under this embedding, we see that $\psi_i(V)$ and $\psi_i(W)$ are indecomposable representations of the quiver $Q_n$. We consider the representation 
\begin{align*}
R &= \psi_1(V) \oplus \psi_2(W) \oplus \psi_3(V) \cdots \\
&=  \bigoplus\limits_{i \text{ odd}} \psi_i(V) \oplus \bigoplus\limits_{i \text{ even}} \psi_i(W).
\end{align*}

We have $\dim(R) = (2,3,3,\dots,3,1)$. Moreover, $R$ is $\sigma$-semistable for the indivisible integral weight $\sigma = (-1,2,-4,8,\dots)$. Since $R$ is a direct sum of indecomposables, it suffices to check $\sigma$-semistability of these indecomposables. That each of these indecomposables is $\sigma$-semistable follows from the above discussion above on $2$-Kronecker quivers. Thus, in particular, $\operatorname{C}(R)$ is non-empty, and $R$ is not in the null cone. 

Moreover, we have that $R$ is a direct sum of $n-1$ indecomposables, and their dimension vectors are linearly independent vectors, and hence it follows from King's stability conditions that $\operatorname{C}(R)$ is at most $1$-dimensional. Since $\operatorname{C}(R)$ is non-empty, and $(-1,2,-4,8,\dots)$ is indivisible, we have that $\operatorname{C}(R) = \Z_{\geq 0} (-1,2,-4,8,\dots)$. More concretely, we have the condition that $\sigma \in \operatorname{C}(R)$, then $\sigma$ is in the kernel of 

$$
\begin{pmatrix}
2 & 1 &  &  & \\
 & 2 & 1 & &  \\
 & & \ddots & \ddots & \\
 & & & 2 & 1 \\
\end{pmatrix}.
$$

The kernel of the above matrix is precisely $\Q$-span of $(-1,2,-4,8,\dots)$, and the smallest integral vector in this $1$-dimensional subspace is $(-1,2,-4,8,\dots)$ by virtue of being indivisible. For the weight $\sigma =  (-1,2,-4,8,\dots)$, we get $|\sigma|_\alpha = 2^n - 2$ by computation. Thus in this case, the semi-invariants of weights $\sigma$ with $|\sigma|_\alpha < 2^n - 2$ do not define the null cone. 

\end{proof}

\begin{remark}
For any given quiver, one might be able to generate stronger bounds by improving the estimates we make at various stages of obtaining our bounds. 
\end{remark}


\begin{thebibliography}{99}

\bibitem{Derksen1} H.~Derksen, {\it Polynomial bounds for rings of invariants}, Proc. Amer. Math. Soc.~{\bf 129} (2001), no.~4, 955--963.
\bibitem{DK} H.~Derksen and G.~Kemper, {\it Computational Invariant Theory.} Invariant Theory and Algebraic Transformation Groups. I. Encyclopaedia of Mathematical Sciences {\bf 130}, Springer-Verlag, 2002.
\bibitem{DM} H.~Derksen and V.~Makam, {\it Polynomial degree bounds for matrix semi-invariants}, {\tt arXiv:1512.03393} [math.RT], 2015.

\bibitem{DW} H.~Derksen and J.~Weyman, {\it Semi-invariants of quivers and saturation of Littlewood-Richardson coefficients}, Journal of the American Math. Soc.~{\bf 13} (2000), 467-479.
\bibitem{DW1}  H.~Derksen and J.~Weyman, {\it On the canonical decomposition of quiver representations}, Compositio Mathematica~{\bf 133} (2002), 245-265.
\bibitem{DW2} H.~Derksen and J.~Weyman, {\it The combinatorics of quiver representations}, Annales de I'institut Fourier~{\bf 61} (2011), 1061-1131.
\bibitem{DKW} J.~Draisma, G.~Kemper and D.~Wehlau, {\it Polarization of separating invariants}, Canad. J. Math.~{\bf 60} (2008), 556-571.

\bibitem{Dom02}M.~Domokos, {\it Finite generating system of matrix invariants}, Math.~Pannon~{\bf 13} (2002), 175--181.
\bibitem{DKZ}M.~Domokos, S.~G.~Kuzmin and A.~N.~Zubkov, {\it Rings of matrix invariants in positive characteristic}, J. of Pure and Applied Algebra~{\bf 176} (2002), 61--80.
\bibitem{DZ} M.~Domokos and A.~N.~Zubkov, {\it Semi-invariants of quivers as determinants}, Transformation groups~{\bf 6} (2001), 9-24.


\bibitem{Formanek}E.~Formanek, {\it Generating the ring of matrix invariants}, in: F.~M.~J.~van Oystaeyen, editor, {\it Ring Theory}, Lecture Notes in mathematics~{\bf 1197}, Springer Berlin Heidelberg, 1986, 73--82.

\bibitem{Haboush}W.~Haboush, {\it Reductive groups are geometrically reductive}, Ann. of Math.~{\bf 102} (1975), 67--85.
\bibitem{Hilbert1}D.~Hilbert, {\it \"Uber die Theorie der algebraischen Formen}, Math. Ann.~{\bf 36} (1890), 473--534.
\bibitem{Hilbert2}D.~Hilbert, {\it \"Uber die vollen Invariantensysteme}, Math. Ann.~{\bf 42} (1893), 313--370.

\bibitem{Kac} V.~Kac, {\it Infinite root systems, representations of graphs and invariant theory. {II}}, Journal of Algebra~{\bf 78} (1982), 141-162.
\bibitem{Kempf} G.~Kempf, {\it The {H}ochster-{R}oberts theorem of invariant theory}, Michigan Math. J.~{\bf 26} (1979), 19-32.
\bibitem{King}A.~D.~King, {\it Moduli of representations of finite-dimensional algebras}, Quart. J. Math. Oxford Ser.~{\bf 45} (1994), no.~180, 515--530.
\bibitem{KP} H.~Kraft and C.~Procesi, {\it Classical Invariant Theory : A primer}, {\tt http://www.unibas.math.ch}.
\bibitem{Knop1} F.~Knop, {\it \"Uber die Glattheit von Quotientenabbildungen}, Manuscripta Math~{\bf 56} (1986), 419-427.
\bibitem{Knop2} F.~Knop, {\it Der Kanonische Modul eines Invariantenringes}, J. Algebra~{\bf 127} (1989), 40-54.


\bibitem{LP} L.~Le~Bruyn and C.~Procesi {\it Semisimple representations of quivers}, Trans. Amer. Math. Soc.~{\bf 317} (1990), no.~2, 585-598.

\bibitem{Nagata}M.~Nagata, {\it Invariants of a group in an affine ring}, J. Math. Kyoto Univ.~{\bf 3} (1963/1964), 369--377.


\bibitem{Popov1}V.~L.~Popov, {\it Constructive Invariant Theory}, Ast\'erique~{\bf 87--88} (1981), 303--334. 
\bibitem{Popov2}V.~L.~Popov, {\it The constructive theory of invariants}, Math. USSR Izvest.~{\bf 10} (1982), 359--376.
\bibitem{Procesi}C.~Procesi, {\it The invariant theory of $n\times n$ matrices}, Adv. in Math.~{\bf 19} (1976), 306--381.
\bibitem{Raz}Y.~Razmyslov, {\it Trace identities of full matrix algebras over a field of characteristic zero}, Comm. in Alg.~{\bf 8} (1980), Math. USSR Izv.~{\bf 8} (1974), 727--760.
\bibitem{SVd}  A.~Schofield and M.~ van der Bergh, {\it Semi-invariants of quivers for arbitrary dimension vectors}, Indag. Mathem., N.S~{\bf 12} (2001), 125--138.
\bibitem{Weyl} H.~Weyl {\it The Classical Groups. Their Invariants and Representations}, Princeton University Press, Princeton, N.J. (1939).
\end{thebibliography}
\end{document}